\documentclass{amsart}

\usepackage{graphicx} 
\usepackage{overpic}

\usepackage{tikz}
\usepackage{tikz-3dplot}
\usetikzlibrary{knots}

\theoremstyle{plain}
\newtheorem{thm}{Theorem}[section]

\newtheorem{lemma}[thm]{Lemma}
\newtheorem{prop}[thm]{Proposition}
\newtheorem{claim}[thm]{Claim}

\theoremstyle{remark}
\newtheorem{remark}[thm]{Remark}

\begin{document}

\title[CAT(0) property and hyperbolicity of affine cactus groups]{CAT(0) property and hyperbolicity of affine cactus groups}

\author{Takatoshi Hama}
\address{Graduate School of Integrated Basic Sciences, Nihon University, 3-25-40 Sakurajosui, Setagaya-ku, Tokyo 156-8550, Japan}
\email{chta23018@g.nihon-u.ac.jp}


\date{\today}

\keywords{Affine cactus group ,Cactus group, CAT(0) group, Hyperbolic group}

\subjclass[2020]{20F65, 20F38, 05E18, 57M60, 20F67}

\begin{abstract}
We show that the affine cactus group is a CAT$(0)$ group for all degrees. 
Furthermore, we show that the affine cactus group $AJ_3$ of degree three is a hyperbolic group. 
\end{abstract}

\maketitle 

\section{Introduction}
The cactus group first appeared in the study of a twisted real form of the Deligne-Mumford space $\overline{M}_{0,n+1}$ of stable genus $0$ curves with $n+1$ marked points \cite{devadoss2000tessellationsmodulispacesmosaic}. 
Furthermore, research on this group from the perspective of geometric group theory has been conducted by Genevois in \cite{genevois2022cactusgroupsviewpointgeometric}. 
Analogously, the affine cactus group was introduced by Ilin, Kamnitzer, Li, Przytycki, and Rybnikov for the study of the moduli space of the cactus flower curves, which serves as an analogue to the aforementioned twisted real form of the Deligne-Mumford space \cite{ilin2024modulispacecactusflower}. 
In addition, Chemin has investigated the affine cactus group from the view point of combinatorial group theory in \cite{chemin2025combinatoricsaffinecactusgroups}. 
In this paper, we focus on the affine cactus group from the perspective of geometric group theory. 
One of the topics in geometric group theory is the study of CAT$(0)$ groups. 
Indeed, Genevois established a criterion equivalent to the statement that a cactus group of any degree is a CAT$(0)$ group. 
Although the CAT$(0)$ property of the affine cactus groups follows from the general theory in \cite{genevois2025graphproductscactusgroups}, we provide a proof here considering the specific structure of the Cayley graph. 
In this paper, we show the following. 
\begin{thm}\label{AJ_CAT0}
    For an integer $n\geq 3$, the affine cactus group $AJ_n$ is a CAT$(0)$ group. 
\end{thm} 
Another major topic in geometric group theory is the hyperbolicity of groups. 
Regarding the hyperbolicity of the cactus groups, it is known that they are hyperbolic for degrees $n \leq 4$, while they are not hyperbolic for $n \geq 5$. 
The second result of this paper clarifies that the affine cactus group of degree $3$ is hyperbolic. 
\begin{thm}\label{AJ_hyp}
    The affine cactus group $AJ_3$ of degree three is a hyperbolic group. 
\end{thm}

\begin{remark}
    In \cite{genevois2022cactusgroupsviewpointgeometric}, it has been shown that the cactus group whose degree is $6$ or more is not a hyperbolic group. 
    The pure cactus group of degree $5$ is known to be the fundamental group of the Deligne-Mumford space $\overline{M}_{0,6}$, which is known to be non-hyperbolic in \cite[Example~3.22]{bonifant2018groupactionsdivisorsplane}. 
    Here, the pure cactus group is the kernel of the natural surjection from the cactus group to the symmetric group of the same degree. 
    This implies that the cactus group of degree $5$ is not a hyperbolic group. 
\end{remark}



\section{Cactus group and affine cactus group}\label{subsec21}
In this section, we recall the purely algebraic definition of the cactus group and the affine cactus group. 

For any integer $n \geq 2$, we define the \textit{cactus group} of degree $n$, denoted by $J_n$, as the group given by the presentation with generators $s_{p,q}$ for $1 \leq p < q \leq n$, subject to the following relations. 
\begin{itemize}
   \item $s_{p,q}^2 = e$ for every $1 \leq p < q \leq n$, 
   \item $s_{p,q}s_{m,r} = s_{m,r}s_{p,q}$ for all $1 \leq p < q \leq n$ and $1 \leq m < r \leq n$ satisfying $[p, q] \cap [m, r] = \emptyset$, 
   \item $s_{p,q}s_{m,r} = s_{p+q-r,p+q-m}s_{p,q}$ for all $1 \leq p < q \leq n$ and $1 \leq m < r \leq n$ satisfying $[m, r] \subset [p, q]$. 
\end{itemize}
Here, $e$ denotes the identity element, and $[p, q]$ denotes the set $\{ p, p+1, \dots, q-1, q \}$ of integers for positive integers $p$ and $q$ with $p < q$. We mainly follow the notation for the cactus groups used in \cite{genevois2022cactusgroupsviewpointgeometric}. 

As a ``cyclic version" of the cactus group, the affine cactus group is defined as follows. 
For any integer $n \geq 2$, we define the \textit{affine cactus group} of degree $n$, denoted by $AJ_n$, as the group given by the presentation with generators $\sigma_{p,q}$ for $1 \leq p \neq q \leq n$, subject to the following relations. 
\begin{itemize}
   \item $\sigma_{p,q}^2 = e$ for every $1 \leq p \neq q \leq n$, 
   \item $\sigma_{p,q}\sigma_{m,r} = \sigma_{m,r} \sigma_{p,q}$ for all $1 \leq p \neq q \leq n$ and $1 \leq m \neq r \leq n$ satisfying $[p, q]_c \cap [m, r]_c = \emptyset$, 
   \item $\sigma_{p,q}\sigma_{m,r} = \sigma_{s_{p,q}(r), s_{p,q}(m)}\sigma_{p,q}$ for all $1 \leq p \neq q \leq n$ and $1 \leq m \neq r \leq n$ satisfying $[m, r]_c \subset [p, q]_c$. 
\end{itemize}

Here, $e$ denotes the identity element, and $[p, q]_c$ denotes the sequence $\{ p,  p+1, \dots, q-1, q \}$ if $p<q$ or $\{ q , q+1, \dots, n-1, n, 1, \dots,p-1, p \}$ if $q<p$ of integers for positive integers $p$ and $q$ with $p \neq q$. 
Furthermore, $s_{p,q}$ is the map given by the following. 
\[
\begin{array}{llc}
   s_{p,q}(r)  &=   \begin{cases}
                    p+q - r +n  & \text{if $r\in [p,q]_c$ and $r \geq p+q$}, \\
                    p+q - r  & \text{if $r \in [p,q]_c$ and $r <p + q$}, \\
                    p & \text{otherwise}. 
                \end{cases}
\end{array}
\]
The notation and terminology used in our definition is based on \cite{chemin2025combinatoricsaffinecactusgroups}. 

For the case of degree three, the affine cactus group $AJ_3$ has the following presentation. 
\[ 
\left\langle 
\sigma_{1,2}, \sigma_{2,3}, \sigma_{3,1}, \sigma_{1,3}, \sigma_{2,1}, \sigma_{3,2} 
\left| 
\begin{array}{l}
\sigma_{1,2}^2= \sigma_{2,3}^2= \sigma_{3,1}^2= \sigma_{1,3}^2= \sigma_{2,1}^2= \sigma_{3,2}^2=e,\\ 
\sigma_{1,2} \sigma_{1,3} = \sigma_{1,3} \sigma_{2,3},\ 
\sigma_{2,3} \sigma_{1,3} = \sigma_{1,3} \sigma_{1,2},\\
\sigma_{2,3} \sigma_{2,1} = \sigma_{2,1} \sigma_{3,1},\ 
\sigma_{3,1} \sigma_{2,1} = \sigma_{2,1} \sigma_{2,3},\\ 
\sigma_{3,1} \sigma_{3,2} = \sigma_{3,2} \sigma_{1,2},\ 
\sigma_{1,2} \sigma_{3,2} = \sigma_{3,2} \sigma_{3,1} 
\end{array}
\right. 
\right\rangle. 
\]

The elements of the affine cactus group can be presented by diagrams of vertical strands on the cylinder. 
Some such diagrams for $AJ_3$ are shown in Figure~\ref{Cyl_gens}. 
\begin{figure}[htb]
    \label{Cyl_gens}
    \centering
\begin{overpic}[width=.9\textwidth]{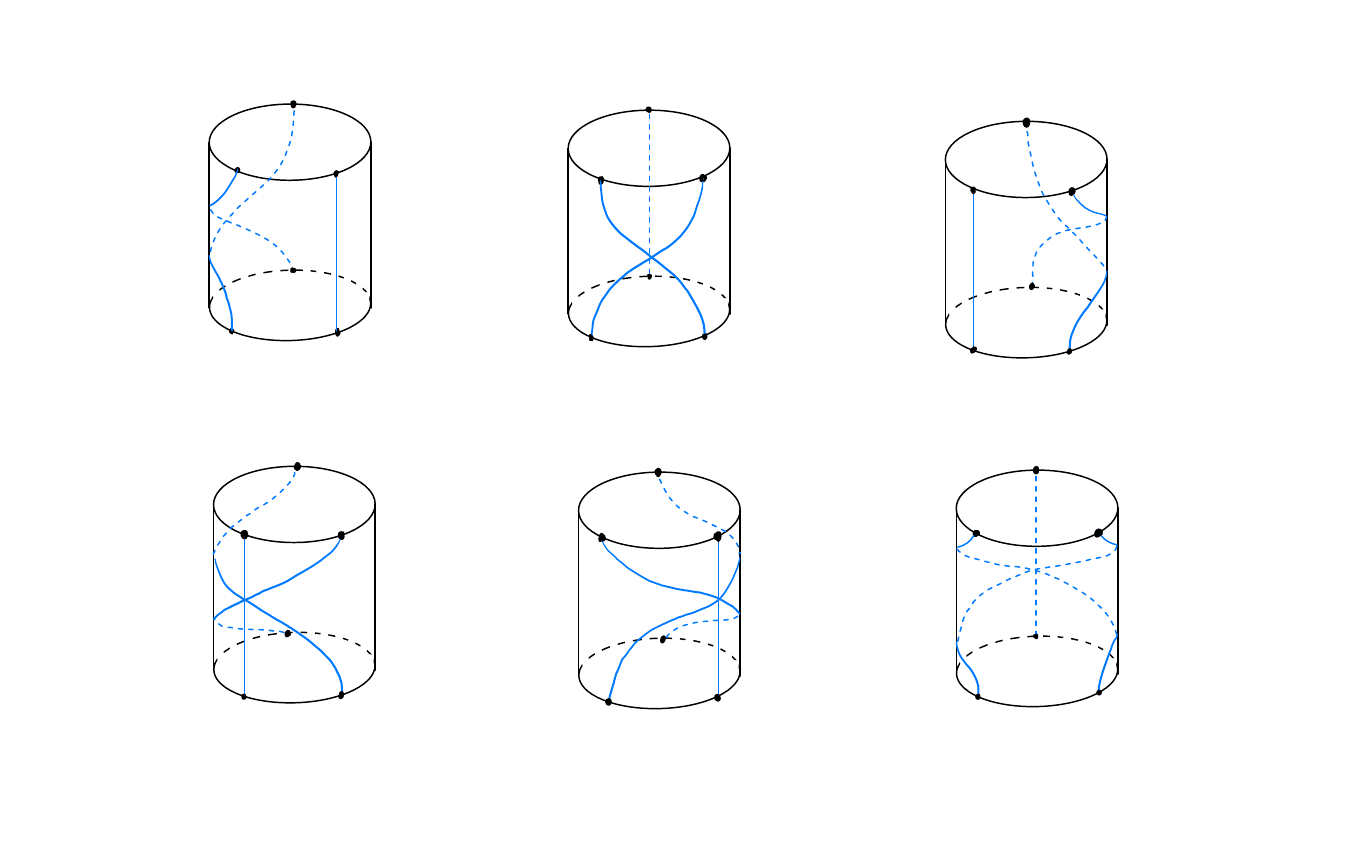}

\put(19,34){$\sigma_{1,2}$}
\put(46,34){$\sigma_{2,3}$}
\put(74,34){$\sigma_{3,1}$}

\put(20,56){1}
\put(13,50){2}
\put(28,50){3}

\put(47,56){1}
\put(40,49){2}
\put(54,49){3}

\put(75,56){1}
\put(67,48){2}
\put(82,48){3}

\put(19,6){$\sigma_{1,3}$}
\put(46,6){$\sigma_{2,1}$}
\put(74,6){$\sigma_{3,2}$}

\put(20,29){1}
\put(13,24){2}
\put(28,24){3}

\put(48,29){1}
\put(40,24){2}
\put(55,24){3}

\put(75,29){1}
\put(68,24){2}
\put(83,24){3}

\end{overpic}
    \caption{The diagrams of generators of $AJ_3$}
\end{figure}
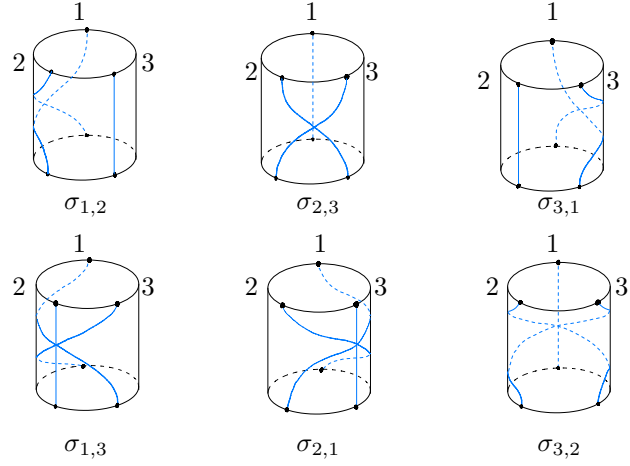
See \cite{chemin2025combinatoricsaffinecactusgroups} for more details. 

As shown in \cite[Proposition 2.6]{genevois2022cactusgroupsviewpointgeometric}, every element of the cactus group has a unique normal form, that is, a uniquely determined word in the presentation. We can show that this is also the case for the affine cactus group in the same way, by cutting the cylinder that supports the diagram of generators where the relations are applied (see Figure~\ref{normal_form}). 

A word $w = \sigma_{p_1, q_1} \sigma_{p_2, q_2} \cdots \sigma_{p_k, q_k}$ representing an element in $AJ_n$ is said to be \textit{in normal form} if it is reduced and satisfies the following condition: 

By applying the commutation and structural relations of $AJ_n$, no generator can be moved further to the left to increase the word's priority, where we prioritize placing generators with smaller starting indices, or those with longer cyclic intervals, as far to the left as possible. More precisely, $w$ is in normal form if it cannot be transformed by:
\begin{itemize}
    \item interchanging commuting disjoint generators to bring the one with a smaller starting index to the left.
    \item Flipping a generator with a longer cyclic interval to the left past a contained generator with a smaller cyclic interval via the relation $\sigma_{s_{p,q}(r), s_{p,q}(m)}\sigma_{p,q} = \sigma_{p,q}\sigma_{m,r}$. 
\end{itemize}

This implies that the word problem for the affine cactus group is solvable. 
We remark that Chemin explicitly shows that the word problem can be solved combinatorially \cite[Corollary~4.1]{chemin2025combinatoricsaffinecactusgroups}. 

\begin{figure}[htb]
    \centering 
    \begin{overpic}[width=1\linewidth]{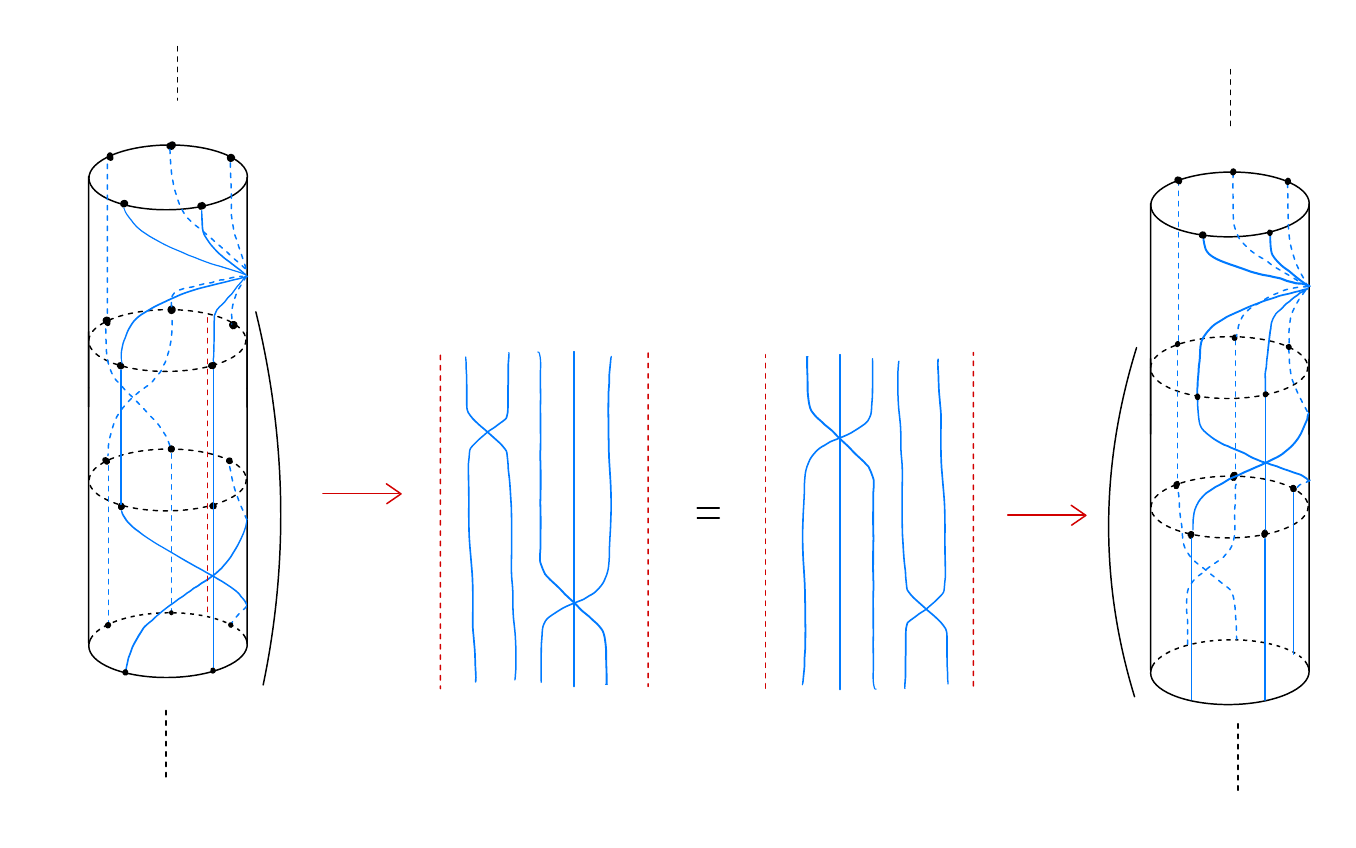}
            \put(7,0.5){$\sigma_{31}\sigma_{12}\sigma_{35}$}
            \put(24,23){cut}

            \put(37,0.5){$\sigma_{12}\sigma_{35}$}
            \put(61,0.5){$\sigma_{35}\sigma_{12}$}

            \put(85,0.5){$\sigma_{31}\sigma_{35}\sigma_{12}$}
            \put(74,22){glue}
            
    \end{overpic}
    \caption{Cutting the cylinder which supports the diagram which are applied the relation.}
    \label{normal_form}
\end{figure}
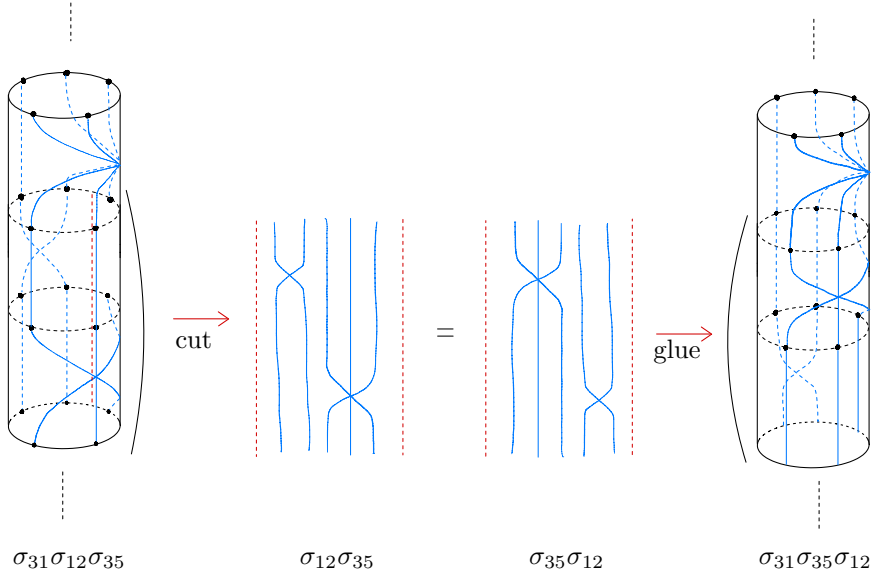


\section{Proof of Theorem~\ref{AJ_CAT0}}\label{sec:prop} 
In this section, we give a proof of Theorem~\ref{AJ_CAT0}. 
Before starting the proof, we prepare some terminology and state the necessary lemma and propositions. 

For the Cayley graph Cay$(G,S)$, the \textit{Cayley complex} $C(G)$ of Cay$(G,S)$ is the cell complex that is obtained by attaching $2$-cells to each cycle representing a relation in Cay$(G,S)$. See \cite[Section~1.3]{AT} for details. 

In the following, a \textit{cube complex} is defined as a CW complex, which is the quotient of disjoint union cubes $[0,1]^{n_\lambda}$ by gluing together cubes by isometries by their faces. See \cite{MR1744486} for details. 
Thus, all $n$-cells are homeomorphic to a cube $[0,1]^n$ for some integer $n$. 

\begin{lemma}\label{AJcube}
    The Cayley complex $C(AJ_n)$ of $AJ_n$ with respect to standard generators is a cube complex. 
\end{lemma}

\begin{proof}
    Since the relations of the standard presentation of $AJ_n$ are given by words of length four, the minimum length of the cycles around $e$ in Cay$(AJ_n,S)$ is $4$. By the left $G$-action, the minimum lengths of the cycles around each vertex are $4$. 
    The definition of the Cayley complex concludes that all $2$-cells in $C(AJ_n)$ are homeomorphic to $I^2$. 
\end{proof}

Next, we recall the definitions of the CAT$(0)$ space, the CAT$(0)$ group, the median graph, and the cube completion. 

A metric space $(X, d)$ is called a \textit{CAT(0) space} if every geodesic triangle $(\gamma_0: [0,L_1] \to X, \gamma_1: [0,L_2] \to X, \gamma_2: [0,L_3] \to X)$ in $X$ satisfies the following condition: 
For a comparison triangle $(\gamma_0' : [0,L_1] \to \mathbb{R}^2, \gamma_1' : [0,L_1] \to \mathbb{R}^2, \gamma_2' : [0,L_2] \to \mathbb{R}^2)$ in the Euclidean plane $(\mathbb{R}^2, d_{\mathbb{R}^2})$ with the same side lengths, for any indices $j, k \in \{0, 1, 2\}$, and for any $s \in [0, L_j]$ and $t \in [0, L_k]$, the following inequality holds. 
\[
d(\gamma_j (s), \gamma_k (t)) \leq d_{\mathbb{R}^2} (\gamma_j' (s), \gamma_k' (t)) 
\]
A group is called a \textit{CAT$(0)$ group} if it admits a properly discontinuous and cocompact action on a non-empty CAT$(0)$ space. 

A graph $X$ is called a \textit{median graph} if for any triple of vertices $x_1$, $x_2$, and $x_3$ in $X$, there exists a unique vertex $m$ such that it satisfies the following. 
\[ \displaystyle
d(x_i, x_j) = d(x_i, m) + d(m, x_j) \ \text{for all $i \neq j$} 
\]
There is a close relationship between the CAT$(0)$ spaces and the median graphs. This relationship is encapsulated in the following propositions. 
\begin{prop}[{\cite[Theorem~2.3]{genevois2022cactusgroupsviewpointgeometric}}]\label{Median_leads_CAT0}
 A graph is median if and only if its cube completion is $CAT(0)$.
\end{prop}
Here, the \textit{cube completion of a median graph} is the cube complex
obtained by filling with cubes all the subgraphs isomorphic
to one-skeletons of cubes. See \cite{genevois2025cat0cubecomplexesreplaced} and \cite[Theorem~6.1]{ChepoiCube} for details. 
\begin{prop}[{\cite[Theorem~2.4]{genevois2022cactusgroupsviewpointgeometric}}]\label{Median_properties}
 Let $X$ be a cube complex. Assume that the following hold: 
\begin{enumerate}
    \item $X$ is simply connected; 
    \item all the squares are embedded; 
    \item two distinct squares never share two consecutive edges; 
    \item a cycle of three squares spans the two-skeleton of a 3-cube. 
\end{enumerate}
    Then the one-skeleton $X^{(1)}$ of $X$ is a median graph. Moreover, every $4$-cycle in $X^{(1)}$ bounds a square in $X$. 
\end{prop}

Here, \textit{a square} is a $2$-cell in a cube complex, and 
\textit{a cycle of three squares} means that a subcomplex consisting of three squares as shown in Figure~\ref{CycleOfThreeSquares}. 








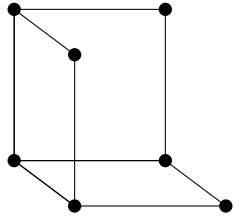
\begin{figure}[htb]\label{CycleOfThreeSquares}
\begin{tikzpicture}
    \def\s{2}    
    \def\dx{-0.8} 
    \def\dy{0.6}  

    \coordinate (B1) at (\dx, \dy);          
    \coordinate (B2) at (\s+\dx, \dy);       
    \coordinate (B3) at (\s+\dx, \s+\dy);    
    \coordinate (B4) at (\dx, \s+\dy);       

    \coordinate (F1) at (0, 0);              
    \coordinate (F2) at (\s, 0);             
    \coordinate (F4) at (0, \s);             


    \draw (F1) -- (F2) -- (B2) -- (B1) -- cycle;

    \draw (B1) -- (B2) -- (B3) -- (B4) -- cycle;

    \draw (F1) -- (B1) -- (B4) -- (F4) -- cycle;

    \foreach \v in {B1, B2, B3, B4, F1, F2, F4} {
        \fill (\v) circle (2.5pt);
    }

\end{tikzpicture}
\caption{Cycle of three squares}
\end{figure}

The proof of Theorem~\ref{AJ_CAT0} is obtained in a similar way as the proof for the cactus group $J_n$ in \cite[Theorem~2.7]{genevois2022cactusgroupsviewpointgeometric}. 

\begin{proof}[Proof of Theorem~\ref{AJ_CAT0}] 
    In general, the left $G$-action of a finitely generated group on itself is properly discontinuous and cocompact. Therefore, the left $G$-action on $C(AJ_n)$ by $AJ$ is properly discontinuous and cocompact. 
    Since the action on $\text{Cay}(AJ,S)$ induces an action on its cube completion, by Lemma~\ref{AJcube} it suffices to show that the Cayley complex $\mathcal{C}(AJ_n)$ of $AJ_n$ satisfies conditions (1), (2), (3), and (4) of Proposition~\ref{Median_properties}. 
    First, by the definition of the Cayley complex, $\mathcal{C}(AJ_n)$ is simply connected. 
    Then $\mathcal{C}(AJ_n)$ satisfies the condition (1). 
    
     Second, we confirm that $C(AJ_n)$ satisfies the condition (2). 
    Consider a square spanned by a $4$-cycle $[e, \sigma_{i,j}, g, \sigma_{k,l}, e]$ that contains $e$ as a vertex. For the square, there are the following three cases. 

    \begin{itemize}
        \item For the vertices $\sigma_{i,j}$ and $\sigma_{k,l}$, $[i,j]_c$ is contained in $[k,l]_c$. 
    
        \item For the vertices $\sigma_{i,j}$ and $\sigma_{k,l}$, $[k,l]_c$ is contained in $[i,j]_c$. 
        
        \item For the vertices $\sigma_{i,j}$ and $\sigma_{k,l}$, $[i,j]_c$ and $[k,l]_c$ are disjoint. 
    \end{itemize}
    
    Since the difference between the first and the second cases is notational, it suffices to consider the first and the third cases without loss of generality. 
    
    In the first case, the relation of $AJ_n$ concludes that the vertex $g$ is equal to $\sigma_{i,j} \sigma_{k,l} = \sigma_{k,l} \sigma_{s_{k,l}(j), s_{k,l}(i)}$, and $\sigma_{k,l} \sigma_{s_{k,l}(j), s_{k,l}(i)}$ is the normal form of $g$. 
    Since all vertices do not coincide with each other, the cycle is embedded. 
    
    In the third case, the relation of $AJ_n$ concludes that the vertex $g$ is equal to $\sigma_{i,j} \sigma_{k,l} = \sigma_{k,l} \sigma_{i,j}$. Without loss of generality, we can assume that $i$, $j$, $k$, $l$ satisfy $i<j<k<l$. In this case, the word $\sigma_{i,j} \sigma_{k,l}$ is the normal form of $g$. 
    Since all vertices do not coincide with each other, the cycle is embedded. 
    
    It concludes that the squares around $e$ are embedded. 
    By the left $G$-action, the squares around $e$ are transported to the squares around all vertices. 
    Therefore, all the squares are embedded, and $\mathcal{C}(AJ_n)$ satisfies the condition (2). 

    We next show that $C(AJ_n)$ satisfies the condition (3). Suppose, for the sake of contradiction, that two distinct $4$-cycles $[e, \sigma_{i,j}, g_1, \sigma_{k,l}, e]$ and $[e, \sigma_{i,j}, g_2, \sigma_{k,l}, e]$ share the consecutive edges $[\sigma_{i,j}, e]$ and $[e, \sigma_{k,l}]$. 
    Since the cycles in a Cayley graph define group relations, the intervals $[i,j]_c$ and $[k,l]_c$ must be nested or disjoint. 
    
    For the case where $[i,j]_c$ and $[k,l]_c$ are nested, without loss of generality, we can assume that $[i,j]_c$ is contained in $[k,l]_c$. The relations around $e$ imply $g_1 = g_2 = \sigma_{k,l} \sigma_{s_{k,l}(j), s_{k,l}(i)}$. 
    
    For the case where $[i,j]_c$ and $[k,l]_c$ are disjoint, 
    similarly, the group relations yield $g_1 = g_2 = \sigma_{k,l} \sigma_{i,j}$. 
    
    In both cases, $g_1 = g_2$, meaning that 
    the two cycles coincide. This contradicts the initial assumption. 
    By the left $G$-action, the squares around $e$ are transported to the squares around all vertices. Thus, each two distinct cycles do not share the consecutive edges. 
    
    Next, let us show that $C(AJ_n)$ satisfies the condition (4). 
    Consider a cycle of squares. We can assume that all squares in the cycle of squares contain $e$ without loss of generality. Thus, let the generators $\sigma_{i,j}$, $\sigma_{k,l}$, and $\sigma_{q,r}$ span a cycle of squares. 
    This assumption means that each two generators of the three involve the relation. 
    Since the difference between generators is notational, it is enough to consider only the following four cases without loss of generality. 
    \begin{itemize}
        \item For the vertices $\sigma_{i,j}$, $\sigma_{k,l}$ and $\sigma_{q,r}$, $[i,j]_c$ contains $[k,l]_c$, $[k,l]_c$ contains $[q,r]_c$. 
        
        \item For the vertices $\sigma_{i,j}$, $\sigma_{k,l}$ and $\sigma_{q,r}$, $[i,j]_c$ contains $[k,l]_c$, $[i,j]_c$ and $[q,r]_c$ are disjoint. 

        \item For the vertices $\sigma_{i,j}$, $\sigma_{k,l}$ and $\sigma_{q,r}$, $[i,j]_c$ contains $[k,l]_c$ and $[q,r]_c$, $[k,l]_c$ and $[q,r]_c$ are disjoint. 

        \item For the vertices $\sigma_{i,j}$, $\sigma_{k,l}$ and $\sigma_{q,r}$,  $[i,j]_c$, $[k,l]_c$, and $[q,r]_c$ are disjoint from each other. 
    \end{itemize}

    We prepare a map $\bar{\ }:  \mathbb{Z} \to [1,n]$, which maps $z$ to $\bar{z}=z + an \in [1,n]$ for some integer $a$, i.e., $\overline{z}$ is equal to $z$ modulo $n$. 
    For the subset $\{ \sigma_{i,j},\ \sigma_{k,l},\ \sigma_{q,r} \}$, that satisfies one of the four conditions above, let $\varphi_{i}$ be a correspondence from $\{ \sigma_{i,j},\ \sigma_{k,l},\ \sigma_{q,r} \}$ to a subset of the standard generators of $J_n$ given by the following: 
    \[
    \sigma_{u,v} \longmapsto s_{\overline{u-i+1},\overline{v-i+1}}
    \]
        


     The map $\varphi_i$ preserves the inclusion property between two intervals of indexes of each two generators, and the next claim follows. 
    \begin{claim}\label{Phi_property} For positive integers $1 \leq i,\ j,\  k,\  l,\  q,\  r \leq n$, the following hold. 
\begin{itemize}
\item  If $[i,j]_c$ contains $[k,l]_c$ and $[k,l]_c$ contains $[q,r]_c$, 
then $[\overline{i-i+1}, \overline{j-i+1}]$ contains $[\overline{k-i+1}, \overline{l-i+1}]$ and $[\overline{k-i+1}, \overline{l-i+1}]$ contains $[\overline{q-i+1}, \overline{r-i+1}]$. 

\item  If $[i,j]_c$ contains $[k,l]_c$, $[i,j]_c$ and $[q,r]_c$ are disjoint, then $[\overline{i-i+1}, \overline{j-i+1}]$ contains $[\overline{k-i+1}, \overline{l-i+1}]$, and $[\overline{i-i+1}, \overline{j-i+1}]$ and $[\overline{q-i+1}, \overline{r-i+1}]$ are disjoint. 

\item  If $[i,j]_c$ contains $[k,l]_c$ and $[q,r]_c$, $[k,l]_c$ and $[q,r]_c$ are disjoint, then $[\overline{i-i+1}, \overline{j-i+1}]$ contains $[\overline{k-i+1}, \overline{l-i+1}]$ and $[\overline{q-i+1}, \overline{r-i+1}]$, and $[\overline{k-i+1}, \overline{l-i+1}]$ and $[\overline{q-i+1}, \overline{r-i+1}]$ are disjoint. 

\item  If $[i,j]_c$, $[k,l]_c$, and $[q,r]_c$ are disjoint from each other, then $[\overline{i-i+1}, \overline{j-i+1}]$, $[\overline{k-i+1}, \overline{l-i+1}]$, and $[\overline{q-i+1}, \overline{r-i+1}]$ are disjoint from each other. 
\end{itemize}

    \end{claim}
    
\begin{proof}
    Since the calculation results depend only on the inclusion relationship between two intervals, it suffices to consider the inclusion relationship between $[i.j]_c$ and $[k,l]_c$. 
    For positive integers $i$, $j$, $k$, and $l$, we first assume that $[i,j]_c$ contains $[k,l]_c$. Then the four possible orderings of $i$, $j$, $k$, and $l$ are $i<k<l<j$ or $j<i<k<l$ or $l<j<i<k$ or $k<l<j<i$. 
    By elementary calculation, we can confirm that $[1,\overline{l-i+1}]$ contains $[\overline{k-i+1},\overline{l-i+1}]$ for each case. Here, we present a proof for the most complicated case that $i$, $j$, $k$, and $l$ satisfy $l<j<i<k$. Actually we obtain the following. 
\[
\begin{array}{l}
    \overline{i-i+1} = 1, \\
    \overline{j-i+1} = j-i+n+1, \\
    \overline{k-i+1} = k-i+1, \\
    \overline{l-i+1} = l-i+n+1. 
\end{array}
\]

    Since $1\leq l < k \leq n$, we see that $k-l<n$. This concludes that $k-i+1 < l-i+n+1$. 
    Therefore, $[1,\overline{l-i+1}]$ contains $[\overline{k-i+1},\overline{l-i+1}]$ in this case. The other cases can be treated in the same way. 
    We secondary assume that $[i,j]_c$ and $[k,l]_c$ are disjoint. 
    Then the four possible orderings of $i$, $j$, $k$, and $l$ are $i<j<k<l$ or $l<i<j<k$ or $k<l<i<j$ or $j<k<l<i$. 
    By an elementary calculation, we can see that $[\overline{1}, \overline{j-i+1}]$ and $[\overline{k-i+1}, \overline{l-i+1}]$ are disjoint for the case $i<j<k<l$. 
    The other cases can be treated in the same way. 
\end{proof}

It concludes that 
if each pair of $\sigma_{i,j}$, $\sigma_{k,l}$, and $\sigma_{q,r}$ make a relation, then each two generators of $\varphi(\sigma_{i,j})$, $\varphi(\sigma_{k,l})$, and $\varphi(\sigma_{q,r})$ involve a relation, i.e., if the vertices $\sigma_{i,j}$, $\sigma_{k,l}$, and $\sigma_{q,r}$ span a cycle of squares, then also $\varphi(\sigma_{i,j})$, $\varphi(\sigma_{k,l})$, and $\varphi(\sigma_{q,r})$ do. 

Let $\psi_{i}$ be a correspondence from the image $\text{Im}\varphi_i = \{ s_{1, \overline{j-i+1}}, s_{\overline{k-i+1},\overline{l-i+1}}, s_{\overline{q-i+1},\overline{r-i+1} } \}$ of $\varphi_{i}$ to a subset of the standard generators of $AJ_n$ given by the following: 
    \[
    s_{u',v'} \longmapsto \sigma_{\overline{u'+i-1}, \overline{v'+i-1}}. 
    \]

Similarly as $\varphi_i$, the map $\psi_i$ also preserves the inclusion property between two intervals of indexes of each two generators, and the next claim follows. 
\begin{claim}\label{Psy_property}
 For positive integers $1 \leq i,\ j,\ k,\ l,\ q,\ r \leq n$ let $i'$, $j'$, $k'$, $l'$, $q'$, $r'$ be such the following. 

\[
\begin{array}{cc}
   i'  & = \overline{i-i+1} \\
   j'  & = \overline{j-i+1} \\
   k'  & = \overline{k-i+1} \\
   l'  & = \overline{l-i+1} \\
   q'  & = \overline{q-i+1} \\
   r'  & = \overline{r-i+1} 
\end{array}
\]
For these integers, the following holds. 
 \begin{itemize}
     \item If $[i',j']$ contains $[k',l']$, and $[k',l']$ contains $[q',r']$, 
     then $[\overline{i'+i-1},\overline{j'+i-1}]_c$ contains $[\overline{k'+i-1},\overline{l'+i-1}]_c$, and then $[\overline{k'+i-1},\overline{l'+i-1}]_c$ contains $[\overline{q'+i-1},\overline{r'+i-1}]_c$. 
     
     \item If $[i',j']$ contains $[k',l']$, and $[k',l']$ and $[q',r']$ are disjoint, 
     then $[\overline{u+i-1},\overline{v+i-1}]_c$ contains $[\overline{k'+i-1},\overline{l'+i-1}]_c$, and then $[\overline{k'+i-1}, \overline{l' +i -1}]_c$ and $[\overline{q'+i-1},\overline{r'+i-1}]_c$ are disjoint. 
    
     \item If $[i',j']$ contains $[k',l']$ and $[q',r']$, and $[u, v]$ and $[q',r']$ are disjoint, 
     then $[\overline{i'+i-1},\overline{j'+i-1}]_c$ contains $[\overline{k'+i-1}, \overline{l'+i-1}]_c$ and $[\overline{q'+i-1}, \overline{r'+i-1}]_c$, and then $[\overline{k'+i-1}, \overline{l'+i-1}]_c$ and $[\overline{q'+i-1}, \overline{r'+i-1}]_c$ are disjoint. 

    \item If $[i',j']$, $[k',l']$, and $[q',r']$ are disjoint from each other, then $[\overline{i'+i-1},\overline{j'+i-1}]_c$, $[\overline{k'+i-1},\overline{l'+i-1}]_c$ and $[\overline{q'+i-1}, \overline{r'+i-1}]_c$ are disjoint from each other. 
 \end{itemize}
 
\end{claim}

\begin{proof}
    Since the calculation result depends only on the inclusion relationship between any two intervals, it suffices to consider the inclusion relationship between $[i',j']$ and $[k', l']$. 
    Suppose that $[i', j']$ contains $[k', l']$, i.e. , $i'<k'<l'<j'$. 
    Consider the twenty four orderings of four integers $i$, $j$, $k$, $l$. Only four orderings can be consistently derived from $i'<k'<l'<j'$: 
    \[
    \begin{array}{cccc}
       i < k < l < j,  & j < i < k < l, & l < j < i < k, & k < l < j < i. 
    \end{array}
    \]
We demonstrate two orderings, namely, the most complicated ones. 
If $k < l < j < i$, the following holds. 
\[
\begin{array}{cl}
   i'  & = \overline{i-i+1} = 1 \\
   j'  & = \overline{j-i+1} = j - i + n + 1 \\
   k'  & = \overline{k-i+1} = k - i + n + 1 \\
   l'  & = \overline{l-i+1} = l - i + n + 1 
\end{array}
\]
\[
\begin{array}{cl}
    \overline{i' +i -1} &= \overline{1 +i -1}  \\
                        &= \overline{i} = i   \\
    \overline{j' +i -1} &= \overline{(j - i + n + 1) +i -1} \\
                        &= \overline{j + n} = j \\
    \overline{k' +i -1} &= \overline{(k - i + n + 1) +i -1} \\
                        &= \overline{k + n} = k \\
    \overline{l' +i -1} &= \overline{ (l - i + n + 1) +i -1} \\
                        &= \overline{l+n} = l
\end{array}
\]
    This clearly concludes that if $[i',j']$ contains $[k', l']$, then $[\overline{i'+i-1},\overline{j'+i-1}]_c$ contains $[\overline{k'+i-1},\overline{l'+i-1}]_c$ for $k < l < j < i$. 
In the same way, we can see that if $[i',j']$ contains $[k', l']$, then $[\overline{i'+i-1},\overline{j'+i-1}]_c$ contains $[\overline{k'+i-1},\overline{l'+i-1}]_c$ for the other orderings.  
    Similarly, we can see that if $[i', j']$ and $[k', l']$ are disjoint, $[i, j]_c$ and $[k, l]_c$ are disjoint. 
\end{proof}

By Claim~\ref{Phi_property}, the three vertices $\sigma_{i,j}$, $\sigma_{k,l}$, and $\sigma_{q,r}$ that span the cycle of three squares around $e$ in $C(AJ_n)$ are corresponded to 
the three vertices $s_{1, \overline{j-i+1}}$, $s_{\overline{k-i+1},\overline{l-i+1}}$, $s_{\overline{q-i+1},\overline{r-i+1}}$ that span a cycle of three squares around $e$ in $C(J_n)$ by $\varphi_{i}$. 
In the proof of \cite{genevois2022cactusgroupsviewpointgeometric}, it is shown that for all the cycles of squares in $C(J_n)$, there exists a unique vertex, say $x$, which spans a cube with the cycle of squares. Here, the normal form of the vertex $x$ consists of the three vertices (generators) that span a cycle of three square, and that is $s_{1, \overline{j-i+1}}, s_{\overline{k-i+1},\overline{l-i+1}}, s_{\overline{q-i+1},\overline{r-i+1} }$. 
By Claim~\ref{Psy_property}, since $\psi_i$ preserves the inclusion property, 
\[
\psi_i ( s_{1, \overline{j-i+1}}) \psi_i (s_{\overline{k-i+1},\overline{l-i+1}}) \psi_i(s_{\overline{q-i+1},\overline{r-i+1}}) =\sigma_{i,j} \sigma_{k,l} \sigma_{q,r}
\]
is the vertex that spans a cube with the cycle of three squares in $C(AJ_n)$. 
Thus $C(AJ_n)$ satisfies the condition (4). 
Consequently, $C(AJ_n)$ is CAT$(0)$, thus $AJ_n$ is a CAT$(0)$ group. 
\end{proof}







\section{Proof of Theorem~\ref{AJ_hyp}}
In this section, we prepare some terminology, propositions, and lemmas which are essential to show Theorem~\ref{AJ_hyp}. 

A \textit{Gromov hyperbolic space} is defined by a geodesic space that satisfies the following. 
There exists a positive real number $\delta$ such that every geodesic triangle $(\gamma_0: [0,L_1] \to X, \gamma_1: [0,L_2] \to X, \gamma_2: [0,L_3] \to X)$ in $X$ satisfies the following condition: 
\[
\begin{array}{lll}
    \gamma_0  &\subset& B_{\delta} (\gamma_1 \cup \gamma_2)  \\
    \gamma_1  &\subset& B_{\delta} (\gamma_0 \cup \gamma_2)  \\
    \gamma_2  &\subset& B_{\delta} (\gamma_0 \cup \gamma_1). 
\end{array}
\]
Here, $B_{\delta}$ is the $\delta$-neighborhood in the metric space. 
\begin{prop}\label{QIandHyp}\cite[Corollary~7.2.13,3]{Geom_gr_th}
    Let $X$ and $Y$ be a geodesic space. 
    Suppose that $X$ and $Y$ are quasi-isometric. Then $X$ is Gromov hyperbolic if and only if $Y$ is Gromov hyperbolic. 
\end{prop}
This proposition means that Gromov hyperbolicity is invariant under quasi-isometry. 

A group $G$ is called a \textit{hyperbolic group} if there exists a generating set $S$ such that the Cayley graph Cay$(G,S)$ is a Gromov hyperbolic space. See \cite{Geom_gr_th} for details on the hyperbolic space and the hyperbolic group. 

\begin{lemma}\label{CayAJ_3_Hyp_plane}
The Cayley graph Cay$(AJ_3, S)$ of $AJ_3$ is quasi-isometric to the hyperbolic plane $\mathbb{H}^2$. 
\end{lemma}

\begin{proof}


The vertices in Cay$(AJ_3,S)$ whose word length is $1$ are the generators; 
    \[
    \begin{array}{cccccc}
        \sigma_{1,2} &
        \sigma_{2,3} &
        \sigma_{3,1} &
        \sigma_{1,3} &
        \sigma_{2,1} &
        \sigma_{3,2}. 
    \end{array}
    \]
The vertices listed below are the elements of word length two with respect to the generators mentioned above. This list is obtained by enumerating all the words of $AJ_3$ whose word length is $2$, and observing normal forms of the elements of $AJ_3$ presented by them. 
    \[
    \begin{array}{cccccc}
        \sigma_{1,2}\sigma_{3,1} & \sigma_{1,2}\sigma_{2,1} & \sigma_{1,2}\sigma_{2,3} & \sigma_{1,3}\sigma_{2,3} & \sigma_{1,3}\sigma_{2,1} & \sigma_{1,3}\sigma_{3,1} \\
        \sigma_{1,3}\sigma_{3,2} & \sigma_{1,3}\sigma_{1,2} & \sigma_{2,3}\sigma_{1,2} & \sigma_{2,3}\sigma_{3,2} & \sigma_{2,3}\sigma_{3,1} & \sigma_{2,1}\sigma_{3,1} \\
        \sigma_{2,1}\sigma_{3,2} & \sigma_{2,1}\sigma_{1,2} & \sigma_{2,1}\sigma_{1,3} &  \sigma_{2,1}\sigma_{2,3} & \sigma_{3,1}\sigma_{2,3} & \sigma_{3,1}\sigma_{1,3} \\
        \sigma_{3,1}\sigma_{1,2} & \sigma_{3,2}\sigma_{1,2} & \sigma_{3,2}\sigma_{1,3} & \sigma_{3,2}\sigma_{2,3} & \sigma_{3,2}\sigma_{2,1} & \sigma_{3,2}\sigma_{3,1} 
    \end{array}
    \]

As shown in Figure~\ref{length_2}, there are edges connecting the above vertices in Cay$(AJ_3,S)$, and the $4$-cycles containing $e$ in Cay$(AJ_3,S)$ are exactly 
\[
\begin{array}{l}
\langle e, \sigma_{1,2}, \sigma_{1,2}\sigma_{1,3}, \sigma_{1,3}, e \rangle, 
\langle e, \sigma_{1,3}, \sigma_{1,3}\sigma_{1,2}, \sigma_{2,3}, e \rangle, 
\langle e, \sigma_{2,3}, \sigma_{2,1}\sigma_{3,1}, \sigma_{2,1}, e \rangle, \\
\langle e, \sigma_{2,1}, \sigma_{2,1}\sigma_{2,3}, \sigma_{3,1}, e \rangle, 
\langle e, \sigma_{3,1}, \sigma_{3,2}\sigma_{1,2}, \sigma_{3,2}, e \rangle, 
\langle e, \sigma_{3,1}, \sigma_{3,2}\sigma_{3,1}, \sigma_{1,2}, e \rangle. 
\end{array}
\]

\begin{figure}[htb]
\label{length_2}
    \centering 

\begin{tikzpicture}
    \def\side{2}; 
    \def\lineLen{1.5}; 

    \coordinate (O) at (0,0); 
    \node[below right, xshift=1mm, yshift=-1mm] at (O) {$e$}; 

    \begin{scope}[rotate=0]
        \coordinate (A) at (0, 0); 
        \coordinate (B) at (\side, 0); 
        \coordinate (C) at (\side * cos{60} + \side, \side * sin{60});  
        \coordinate (D) at (\side * cos{60}, \side * sin{60}); 
        \draw (A) -- (B) -- (C) -- (D) -- cycle; 
        \node[right, xshift=1mm] at (B) {$\sigma_{1,2}$}; 
        \node[above right, xshift=1mm, yshift=1mm] at (C) {$\sigma_{1,3}\sigma_{2,3}$}; 
        \coordinate (end1) at ($ (B) + (0:\lineLen) $);
        \coordinate (end2) at ($ (B) + (30:\lineLen) $);
        \coordinate (end3) at ($ (B) + (-30:\lineLen) $);
        \draw (B) -- (end1);
        \draw (B) -- (end2);
        \draw (B) -- (end3);
        \node[right, xshift=1mm] at (end1) {$\sigma_{1,2}\sigma_{2,1}$}; 
        \node[above right, xshift=1mm, yshift=1mm] at (end2) {$\sigma_{1,2}\sigma_{2,3}$}; 
        \node[below right, xshift=1mm, yshift=-1mm] at (end3) {$\sigma_{1,2}\sigma_{3,1}$}; 
    \end{scope}

    \begin{scope}[rotate=60]
        \coordinate (A) at (0, 0);
        \coordinate (B) at (\side, 0);
        \coordinate (C) at (\side * cos{60} + \side, \side * sin{60});
        \coordinate (D) at (\side * cos{60}, \side * sin{60});
        \draw (A) -- (B) -- (C) -- (D) -- cycle;
        \node[above right, xshift=1mm, yshift=1mm] at (B) {$\sigma_{1,3}$}; 
        \node[above, yshift=1.4mm] at (C) {$\sigma_{1,3}\sigma_{1,2}$};
        \coordinate (end1) at ($ (B) + (0:\lineLen) $); 
        \coordinate (end2) at ($ (B) + (30:\lineLen) $); 
        \coordinate (end3) at ($ (B) + (-30:\lineLen) $); 
        \draw (B) -- (end1);
        \draw (B) -- (end2);
        \draw (B) -- (end3);
        \node[above right, xshift=1mm, yshift=1mm] at (end1) {$\sigma_{1,3}\sigma_{3,1}$}; 
        \node[above, yshift=1mm] at (end2) {$\sigma_{1,3}\sigma_{3,2}$}; 
        \node[right, xshift=1mm] at (end3) {$\sigma_{1,3}\sigma_{2,1}$}; 
    \end{scope}

    \begin{scope}[rotate=120]
        \coordinate (A) at (0, 0);
        \coordinate (B) at (\side, 0);
        \coordinate (C) at (\side * cos{60} + \side, \side * sin{60});
        \coordinate (D) at (\side * cos{60}, \side * sin{60});
        \draw (A) -- (B) -- (C) -- (D) -- cycle;
        \node[above left, xshift=-1mm, yshift=1mm] at (B) {$\sigma_{2,3}$}; 
        \node[above left, xshift=-1mm, yshift=1mm] at (C) {$\sigma_{2,1}\sigma_{3,1}$}; 
        \coordinate (end1) at ($ (B) + (0:\lineLen) $);
        \coordinate (end2) at ($ (B) + (30:\lineLen) $);
        \coordinate (end3) at ($ (B) + (-30:\lineLen) $);
        \draw (B) -- (end1);
        \draw (B) -- (end2);
        \draw (B) -- (end3);
        \node[above left, xshift=-1mm, yshift=-1mm] at (end1) {$\sigma_{2,3}\sigma_{3,2}$}; 
        \node[left, xshift=-1mm] at (end2) {$\sigma_{23}\sigma_{31}$}; 
        \node[below left, xshift=3mm, yshift=4mm] at (end3) {$\sigma_{2,3}\sigma_{1,2}$}; 
    \end{scope}

    \begin{scope}[rotate=180]
        \coordinate (A) at (0, 0);
        \coordinate (B) at (\side, 0);
        \coordinate (C) at (\side * cos{60} + \side, \side * sin{60});
        \coordinate (D) at (\side * cos{60}, \side * sin{60});
        \draw (A) -- (B) -- (C) -- (D) -- cycle;
        \node[left, xshift=-1mm] at (B) {$\sigma_{2,1}$}; 
        \node[below left, xshift=-1mm, yshift=-1mm] at (C) {$\sigma_{2,1}\sigma_{2,3}$}; 
        \coordinate (end1) at ($ (B) + (0:\lineLen) $);
        \coordinate (end2) at ($ (B) + (30:\lineLen) $);
        \coordinate (end3) at ($ (B) + (-30:\lineLen) $);
        \draw (B) -- (end1);
        \draw (B) -- (end2);
        \draw (B) -- (end3);
        \node[left, xshift=-1mm] at (end1) {$\sigma_{2,1}\sigma_{1,2}$}; 
        \node[below left, xshift=-1mm, yshift=-1mm] at (end2) {$\sigma_{2,1}\sigma_{1,3}$}; 
        \node[above left, xshift=-1mm, yshift=1mm] at (end3) {$\sigma_{2,1}\sigma_{3,2}$}; 
    \end{scope}

    \begin{scope}[rotate=240]
        \coordinate (A) at (0, 0);
        \coordinate (B) at (\side, 0);
        \coordinate (C) at (\side * cos{60} + \side, \side * sin{60});
        \coordinate (D) at (\side * cos{60}, \side * sin{60});
        \draw (A) -- (B) -- (C) -- (D) -- cycle;
        \node[below left, xshift=-1mm, yshift=-1mm] at (B) {$\sigma_{3,1}$}; 
        \node[below, yshift=-1.4mm] at (C) {$\sigma_{3,2}\sigma_{1,2}$}; 
        \coordinate (end1) at ($ (B) + (0:\lineLen) $);
        \coordinate (end2) at ($ (B) + (30:\lineLen) $);
        \coordinate (end3) at ($ (B) + (-30:\lineLen) $);
        \draw (B) -- (end1);
        \draw (B) -- (end2);
        \draw (B) -- (end3);
        \node[below left, xshift=-1mm, yshift=-1mm] at (end1) {$\sigma_{3,1}\sigma_{1,3}$}; 
        \node[below, yshift=-1mm] at (end2) {$\sigma_{3,1}\sigma_{1,2}$}; 
        \node[left, xshift=-1mm] at (end3) {$\sigma_{3,1}\sigma_{2,3}$}; 
    \end{scope}

    \begin{scope}[rotate=300]
        \coordinate (A) at (0, 0);
        \coordinate (B) at (\side, 0);
        \coordinate (C) at (\side * cos{60} + \side, \side * sin{60});
        \coordinate (D) at (\side * cos{60}, \side * sin{60});
        \draw (A) -- (B) -- (C) -- (D) -- cycle;
        \node[below right, xshift=1mm, yshift=-1mm] at (B) {$\sigma_{3,2}$}; 
        \node[below right, xshift=1mm, yshift=-1mm] at (C) {$\sigma_{3,2}\sigma_{3,1}$}; 
        \coordinate (end1) at ($ (B) + (0:\lineLen) $);
        \coordinate (end2) at ($ (B) + (30:\lineLen) $);
        \coordinate (end3) at ($ (B) + (-30:\lineLen) $);
        \draw (B) -- (end1);
        \draw (B) -- (end2);
        \draw (B) -- (end3);
        \node[below right, xshift=1mm, yshift=-1mm] at (end1) {$\sigma_{3,2}\sigma_{2,3}$}; 
        \node[right, xshift=1mm] at (end2) {$\sigma_{3,2}\sigma_{2,1}$}; 
        \node[above right, xshift=-2mm, yshift=-4mm] at (end3) {$\sigma_{3,2}\sigma_{1,3}$}; 
    \end{scope}
\end{tikzpicture}

    \caption{Neighborhood of $e$ in the Cayley complex $\mathcal{C}(AJ_3,S)$}
\end{figure}
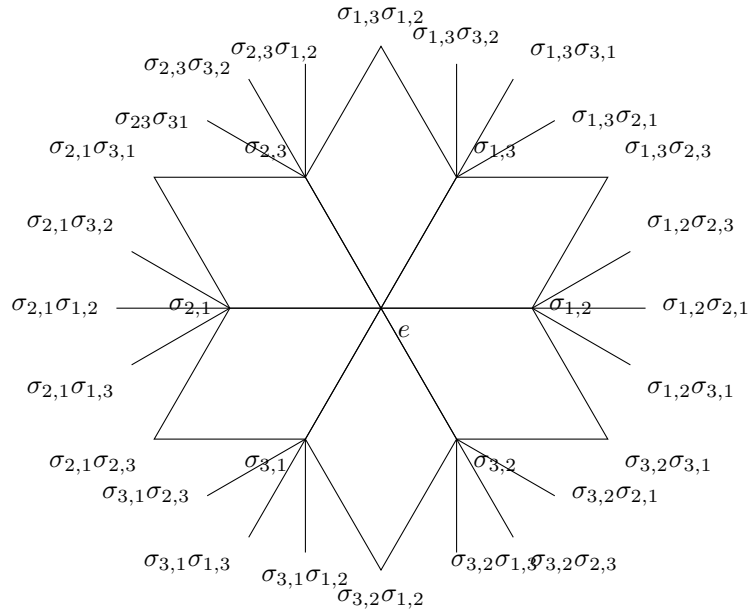

This means that $\mathcal{C}(AJ_3,S)$ is a cell complex where each vertices are surrounded by six quadrilateral $4$-cycles. 

Let $f$ be an embedding which corresponds each $4$-cycle on Cay$(AJ_3,S)$ to the boundary of a square (regular quadrilateral) in the hyperbolic plane $\mathbb{H}^2$ whose internal angles are equal to $\frac{ \pi}{3}$. 
The image of the map $f$ gives a tiling of $\mathbb{H}^2$ of type $\{4,6\}$. 
See Figure~\ref{fig:46tesse}. 

\begin{figure}[htb]
\label{fig:46tesse}
\includegraphics[width=.6\textwidth]{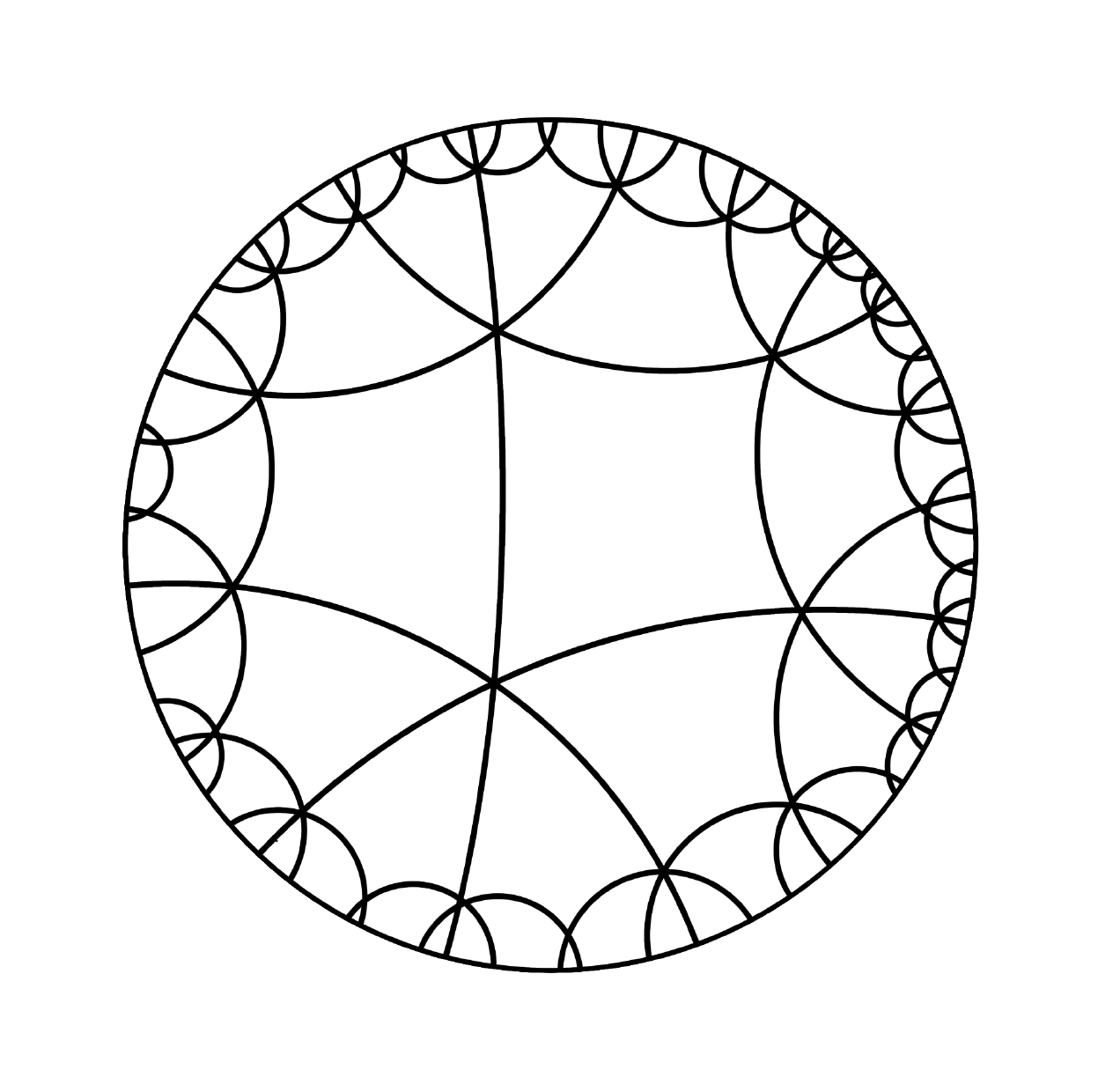}
\caption{The $\{4,6\}$-tesselation of the hyperbolic plane. }
\end{figure}

Let $x$ and $y$ be vertices in Cay$(AJ_3,S)$, and $\beta$ be the geodesic connecting $x$ to $y$. Then, $f(\beta)$ is a polyline thet connects $f(x)$ to $f(y)$ by the shortest distance. 
Let $\alpha$ be the geodesic in $\mathbb{H}^2$ which connects $f(x)$ and $f(y)$. 
See Figure~\ref{polyline} for example. 

\begin{figure}[htb]
\label{polyline}
    \centering
        \begin{overpic}[width=.6\textwidth]{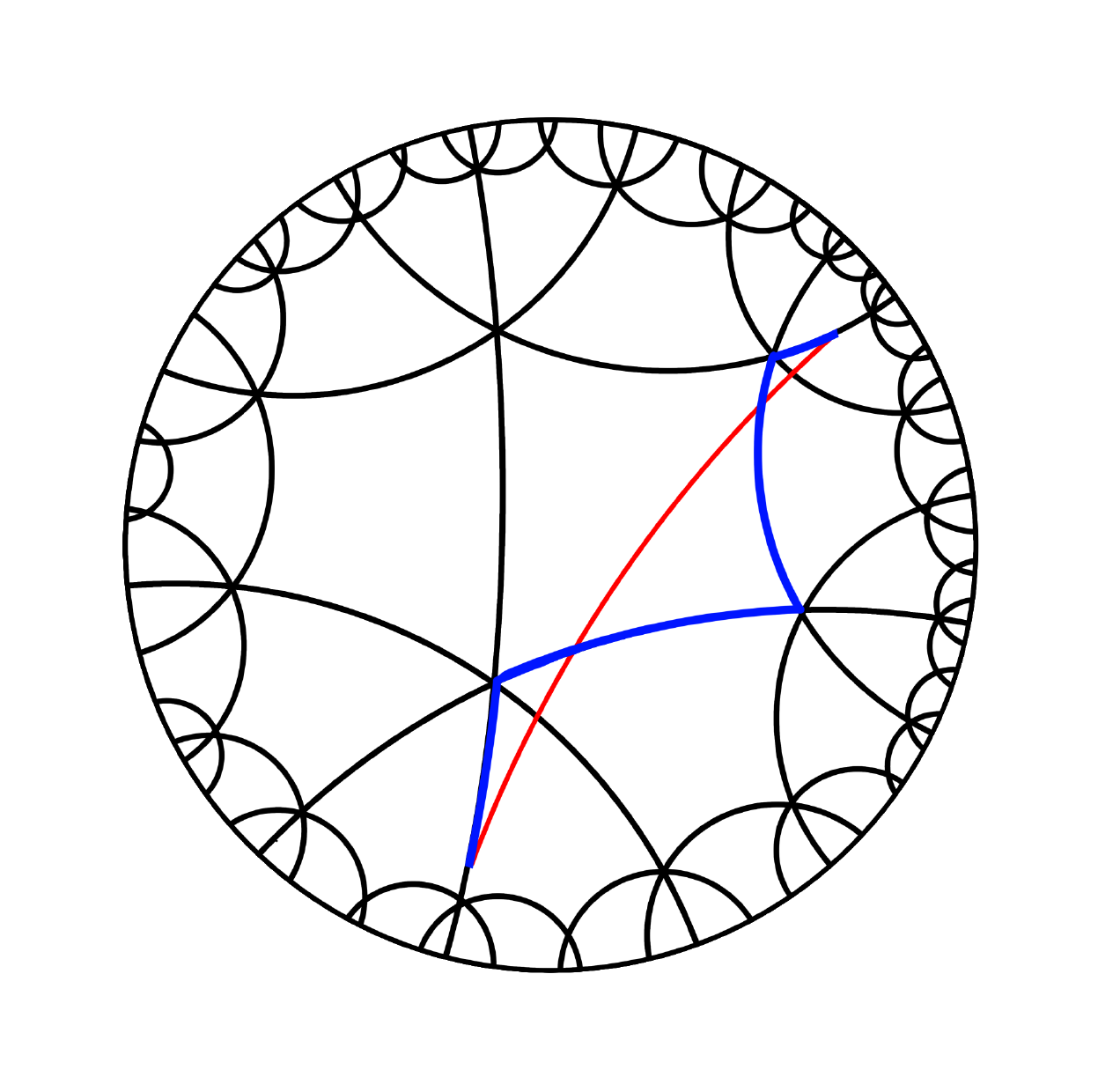}
            \put(65,38){$\beta$}
            \put(58,55){$\alpha$}
            \put(35,25){\small{$f(x)$}}
            \put(73,65){\footnotesize{$f(y)$}}
        \end{overpic}
\caption{Geodesic and polyline connecting $f(x)$ to $f(y)$.}
\end{figure}
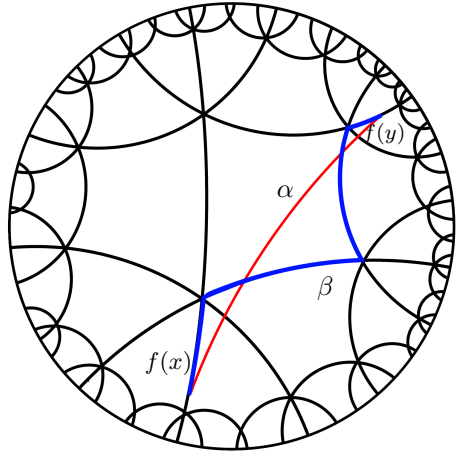


Then, by direct hyperbolic geometrical calculation, we can find that $\alpha$ is quasi-isometric to $\beta$. 
Since $\beta$ is isometric to $f(\beta)$, then $\beta$ is quasi-isometric to $\alpha$. 
This concludes that $f$ is a quasi-isometric embedding. 
Therefore, 
Cay$(AJ_3, S)$ is quasi-isometric to $\mathbb{H}^2$. 

\end{proof}

\begin{proof}[Proof of Theorem~\ref{AJ_hyp}]
    By Lemma~\ref{CayAJ_3_Hyp_plane}, Cay$(AJ_3,S)$ is quasi-isometric to $\mathbb{H}^2$. 
    Since $\mathbb{H}^2$ is Gromov hyperbolic, 
    Proposition~\ref{QIandHyp} concludes that Cay$(AJ_3,S)$ is a Gromov hyperbolic space. 
    By the definition of the hyperbolic group, $AJ_3$ is a hyperbolic group. 
\end{proof}








\section*{Acknowledgments}
The author would like to thank his supervisor professor Kazuhiro Ichihara. He also would like to thank Carl-Fredrik Nyberg Brodda for inspiring this research and thanks Anthony Genevois for letting him know about the preprint \cite{genevois2025graphproductscactusgroups}. 

\bibliographystyle{amsplain}
\bibliography{ref}

@misc{devadoss2000tessellationsmodulispacesmosaic,
      title={Tessellations of Moduli Spaces and the Mosaic Operad}, 
      author={Satyan L. Devadoss},
      year={2000},
      eprint={math/9807010},
      archivePrefix={arXiv},
      primaryClass={math.AG},
      url={https://arxiv.org/abs/math/9807010}, 
}

@article {genevois2022cactusgroupsviewpointgeometric,
    AUTHOR = {Genevois, Anthony},
     TITLE = {Cactus groups from the viewpoint of geometric group theory},
   JOURNAL = {Topology Proc.},
  FJOURNAL = {Topology Proceedings},
    VOLUME = {66},
      YEAR = {2025},
     PAGES = {59--103},
      ISSN = {0146-4124,2331-1290},
   MRCLASS = {20F65 (20F10 20F67)},
  MRNUMBER = {4874027},
}

@book {AT,
    AUTHOR = {Hatcher, Allen},
     TITLE = {Algebraic topology},
 PUBLISHER = {Cambridge University Press, Cambridge},
      YEAR = {2002},
     PAGES = {xii+544},
      ISBN = {0-521-79160-X; 0-521-79540-0},
   MRCLASS = {55-01 (55-00)},
  MRNUMBER = {1867354},
MRREVIEWER = {Donald\ W.\ Kahn},
}

@misc{chemin2025combinatoricsaffinecactusgroups,
      title={Combinatorics of affine cactus groups}, 
      author={Hugo Chemin},
      year={2025},
      eprint={2501.16270},
      archivePrefix={arXiv},
      primaryClass={math.CO},
      url={https://arxiv.org/abs/2501.16270}, 
}

@misc{ilin2024modulispacecactusflower,
      title={The moduli space of cactus flower curves and the virtual cactus group}, 
      author={Aleksei Ilin and Joel Kamnitzer and Yu Li and Piotr Przytycki and Leonid Rybnikov},
      year={2024},
      eprint={2308.06880},
      archivePrefix={arXiv},
      primaryClass={math.AG},
      url={https://arxiv.org/abs/2308.06880}, 
}

@book {Geom_gr_th,
    AUTHOR = {L\"oh, Clara},
     TITLE = {Geometric group theory},
    SERIES = {Universitext},
      NOTE = {An introduction},
 PUBLISHER = {Springer, Cham},
      YEAR = {2017},
     PAGES = {xi+389},
      ISBN = {978-3-319-72253-5; 978-3-319-72254-2},
   MRCLASS = {20-01 (20F65 53C23 57M07)},
  MRNUMBER = {3729310},
       DOI = {10.1007/978-3-319-72254-2},
       URL = {https://doi.org/10.1007/978-3-319-72254-2},
}

@misc{bonifant2018groupactionsdivisorsplane,
      title={Group Actions, Divisors, and Plane Curves}, 
      author={Araceli Bonifant and John Milnor},
      year={2018},
      eprint={1809.05191},
      archivePrefix={arXiv},
      primaryClass={math.AG},
      url={https://arxiv.org/abs/1809.05191}, 
}

@book {MR1744486,
    AUTHOR = {Bridson, Martin R. and Haefliger, Andr\'e},
     TITLE = {Metric spaces of non-positive curvature},
    SERIES = {Grundlehren der mathematischen Wissenschaften [Fundamental
              Principles of Mathematical Sciences]},
    VOLUME = {319},
 PUBLISHER = {Springer-Verlag, Berlin},
      YEAR = {1999},
     PAGES = {xxii+643},
      ISBN = {3-540-64324-9},
   MRCLASS = {53C23 (20F65 53C70 57M07)},
  MRNUMBER = {1744486},
MRREVIEWER = {Athanase\ Papadopoulos},
       DOI = {10.1007/978-3-662-12494-9},
       URL = {https://doi.org/10.1007/978-3-662-12494-9},
}

@misc{genevois2025cat0cubecomplexesreplaced,
      title={Why CAT(0) cube complexes should be replaced with median graphs},
      author={Anthony Genevois},
      year={2025},
      eprint={2309.02070},
      archivePrefix={arXiv},
      primaryClass={math.GR},
      url={https://arxiv.org/abs/2309.02070},
}

@article {ChepoiCube,
    AUTHOR = {Chepoi, Victor},
     TITLE = {Graphs of some {${\rm CAT}(0)$} complexes},
   JOURNAL = {Adv. in Appl. Math.},
  FJOURNAL = {Advances in Applied Mathematics},
    VOLUME = {24},
      YEAR = {2000},
    NUMBER = {2},
     PAGES = {125--179},
      ISSN = {0196-8858,1090-2074},
   MRCLASS = {57M07 (05C10 20F67)},
  MRNUMBER = {1748966},
       DOI = {10.1006/aama.1999.0677},
       URL = {https://doi.org/10.1006/aama.1999.0677},
}

@misc{genevois2025graphproductscactusgroups,
      title={Beyond graph products and cactus groups: quandle products of groups}, 
      author={Anthony Genevois},
      year={2025},
      eprint={2505.09194},
      archivePrefix={arXiv},
      primaryClass={math.GR},
      url={https://arxiv.org/abs/2505.09194}, 
}

\end{document}